\documentclass[11pt,twoside]{article}
\usepackage{mathtools}
\usepackage{amsmath,amssymb,amsthm}
\usepackage{color}

\numberwithin{equation}{section}
\newtheorem{theorem}{Theorem}[section]
\newtheorem{proposition}[theorem]{Proposition}
\newtheorem{lemma}[theorem]{Lemma}
\newtheorem{assumptionF}{Assumption}

\newtheorem{assumptionT}{Assumption}

\theoremstyle{remark}
\newtheorem{remark}[theorem]{Remark}
\newtheorem{example}[theorem]{Example}

\newcommand{\bke}[1]{\left ( #1 \right )}
\DeclarePairedDelimiter{\norm}{\lVert}{\rVert}
\DeclarePairedDelimiter{\bka}{\langle}{\rangle}
\newcommand{\N}{\mathbb{N}}
\newcommand{\R}{\mathbb{R}}
\newcommand{\C}{\mathbb{C}}

\newcommand{\ntt}{N([t,\infty))}
\newcommand{\stt}{S([t,\infty))}
\newcommand{\sourceterm}{H}
\newcommand{\profile}{W}
\newcommand{\tsum}{{\textstyle \sum}}
\newcommand{\om}{\omega}
\newcommand{\I}{\infty}
\newcommand{\ab}{{\tilde {\alpha}}}
\newcommand{\lec}{\lesssim}
\newcommand{\e}{\epsilon}

\title{Infinite soliton and kink-soliton trains \\for nonlinear Schr\"odinger equations}
\author{
Stefan Le Coz,%
\thanks{Institut de Math\'ematiques de Toulouse, Universit\'e Paul Sabatier, 118 route de Narbonne, 31062 Toulouse Cedex 9, France, E-mail address: \texttt{slecoz@math.univ-toulouse.fr}}
\and Tai-Peng Tsai%
\thanks{
Department of Mathematics, University of British Columbia, Vancouver BC
Canada V6T 1Z2, E-mail address: \texttt{ttsai@math.ubc.ca}}
}
\date{\today}

\begin{document}
\maketitle

\begin{abstract}
We look for solutions to generic nonlinear Schr\"odinger equations build upon solitons and kinks. Solitons are localized solitary  waves and kinks are their non localized counter-parts. We prove the existence of \emph{infinite soliton trains}, i.e. solutions behaving at large time as the sum of infinitely many solitons. We also show that one can attach a kink at one end of the train. Our proofs proceed by fixed point arguments around the desired profile. We present two approaches leading to different results, one based on a combination of $L^p-L^{p'}$ dispersive estimates and Strichartz estimates, the other based only on Strichartz estimates. 

\smallskip
\noindent{\it Keywords:}\quad soliton train, multi-soliton, multi-kink,
nonlinear Schr\"odinger equations.

\smallskip
\noindent{\it 2010 Mathematics Subject Classification:} 35Q55(35C08,35Q51).

\end{abstract}

\section{Introduction} \label{sec1}

We consider the nonlinear
Schr\"odinger equation 
\begin{equation}\label{v100a0}\tag{\textsc{nls}}
i\partial_t u +\Delta u  +f(u)=0, 
\end{equation}
where $u=u(t,x)$ is a complex-valued function on $\mathbb R \times
\mathbb R^d$, $d\ge 1$.

Our goal in this paper is push forward a study initiated in \cite{LLT1} on the existence of exotic solutions to  \eqref{v100a0}. We look for
\emph{infinite soliton trains}, i.e. solutions which
behave asymptotically as the sum of infinitely many solitons, possibly attached to a kink at one end.  
We want to show that such a behavior is possible for general nonlinearities under mild hypotheses. A typical
nonlinearity example is the double-power nonlinearity
\begin{equation}
\label{f-pq}
f(u) = |u|^{\alpha}u - |u|^\beta u, \quad 0<\alpha < \beta < \alpha_{\rm max}.
\end{equation}
Here and thereafter we denote the critical exponent by $\alpha_{\max} = +\infty$ for
  $d=1,2$ and $\alpha_{\max}= \frac 4 {d-2}$ for $d\ge 3$.

Let us shortly review some results on multi-solitons, i.e. solutions to \eqref{v100a0} behaving at large time as a finite sum of solitons. The inverse scattering transform provides a convenient way to build multi-solitons (see e.g. \cite{ZaSh72}), however it is limited to integrable equations (for Schr\"odinger equations, only the 1D cubic case is integrable). For non-integrable Schr\"odinger equations, one of the first result of existence of multi-solitons was obtained by Merle \cite{Me90} for $L^2$-critical equations, triggering a series of work on multi-solitons. For energy-subcritical nonlinearities,  C\^ote, Martel and Merle \cite{CoMaMe11,MaMe06} obtained the existence of multi-solitons build upon ground states, while the excited states case was treated by C\^ote and Le Coz \cite{CoLC11} under a high speed assumption. Stability/instability results have been obtained by C\^ote and Le Coz \cite{CoLC11}, Martel, Merle, Tsai \cite{MaMeTs06} and Perelman \cite{Pe04}. However, stability of multi-solitons for power-type nonlinearities is still an open issue. 

 The existence of objects like infinite soliton trains is of importance as they usually provide examples of extreme phenomena in the asymptotic behavior of solutions of nonlinear dispersive equations. For example, for the Korteweg-de Vries equation, an infinite train of solitons was used in \cite{MaMe05} as a counter example to show the optimality of an asymptotic stability statement. For nonlinear Schr\"odinger equation, 
the asymptotic stability results usually  hold under assumptions (typically in weighted spaces) excluding the infinite train behavior. To our knowledge, our previous work \cite{LLT1} was the first one to establish the existence of infinite soliton trains for non-integrable Schr\"odinger equations (for the integrable 1D cubic nonlinear Schr\"odinger equation, the existence of   infinite soliton trains  may be obtained via the inverse-scattering transform, see \cite{Ka95}).

Before stating our main results, let us give some preliminaries.
To work in an energy subcritical context, we first assume the following.

\setcounter{assumptionF}{-1}
\begin{assumptionF}\label{F0}
 Let $d \ge 1$. Suppose $f(u)=g(|u|^2)u$
where  $g\in
C^0([0,\infty),\R)\cap C^2((0,\infty),\R)$, $g(0)=0$ and
\begin{align*} 
|sg'(s)|+|s^2 g^{\prime\prime}(s)| \le
C_0(s^{\alpha_1/2}+s^{\alpha_2/2}), \quad \forall s >0,
\end{align*}
where $0<\alpha_1 \le \alpha_2 < \alpha_{\max}$ and $C_0>0$.  
\end{assumptionF}

A \emph{ bound state} is a nontrivial solution $\phi \in H^1(\R^d)$ of the
elliptic equation
\begin{equation}
\label{soliton.eq}  
\Delta \phi + f(\phi) = \om \phi
\end{equation}
for some frequency $\om>0$. We shall sometimes denote a bound state along with its frequency $(\phi,\omega)$ to emphasize the dependency of $\phi$ on $\omega$.  Any bound state $\phi$ with frequency
$\om$ and parameters $x^0\in \R^d$ (position), $v\in \R^d$ (velocity)
and $\gamma\in \R$ (phase) corresponds to a \emph{solitary wave} solution
(\emph{soliton}) of \eqref{v100a0},
\begin{equation}
\label{soliton.formula}  
R_{\phi, \om, x^0, v, \gamma}(t,x) =e^{i(\omega t + \frac 12 vx -
  \frac 14 |v|^2 t + \gamma)} \phi(x-x^0 -vt).
\end{equation}
The profile of an infinite soliton train is a sum of the form
\begin{equation}
\label{Rj.def}
R_\I= \sum_{j=1}^\infty R_j, \quad
R_j(t,x) = R_{\phi_j, \om_j, x_j^0, v_j, \gamma_j}(t,x), \quad j\in \N,
\end{equation}
where $(R_j)_j$ are given solitons with bound states profiles $(\phi_j, \om_j)$ and parameters
$x^0_j,v_j\in \R^d$ and $\gamma_j\in \R$.  A solution $u(t)$ is
called an \emph{infinite soliton train} if, for some profile $R_\I$,
\[
u(t) - R_\I(t) \to 0\quad \text{ as }\quad t\to \I
\] 
in some space-time norm.

Constructing a solution to \eqref{v100a0} around an infinite train profile as \eqref{Rj.def} is much trickier than when the profile is made with a finite number of solitons. First of all, we need to make sure that the profile is well defined, as the addition of infinitely many solitons may very well be infinite. We also have to take into account that it is very likely that the profile will not belong to the same  functional spaces as the solitons.
In order to deal with these issues we need a control on the growth of the solitons' profiles (see \eqref{exp.decay}) and also to guarantee some space integrability of the train (see \eqref{Ass1}).

We will assume the following for our infinite
train.

\begin{assumptionT}\label{T1}
For $0<\alpha_1<\alpha_{\rm max}$ given, the sequence of bound states $\{(\phi_j,
\om_j): j \in \N\}$ satisfies, for some $0<a<1$ and $D_a$ independent of $j$,
\begin{equation}
\label{exp.decay}
|\phi_j(x)| + \om_j^{-1/2}|\nabla \phi_j(x)| \le D_a \om_j^{1/\alpha_1}
e^{- a \om_j^{1/2} |x|} ,\quad \forall x\in \R^d,\ \forall j \in \N,
\end{equation}
and, for some $r_0\ge 1$, $\frac {d\alpha_1}2< r_0 < 2+\alpha_1$,
\begin{equation}
\label{Ass1}
 A_1:= \sum_{j\in \N} \om_j^{\frac {1}{\alpha_1} - \frac d{2r_0}}< \infty.
\end{equation}
\end{assumptionT}

We say a nonlinearity $f$ satisfies \ref{T1} if such an infinite sequence $(\phi_j,
\om_j)_j$ exists for some $r_0$. Examples of such nonlinearities will be given in Section \ref{sec2.1}.

Note that the set $[1,\I) \cap (\frac {d\alpha_1}2, 2+\alpha_1)$ for $r_0$
  is nonempty since $0<\alpha_1<\alpha_{\rm max}$.  The condition
  $r_0>\frac {d\alpha_1}2$ ensures that the exponent $\frac {1}{\alpha_1} -
  \frac d{2r_0}>0$. Thus $\om_j \to 0$ as $j \to \infty$, and
  \eqref{Ass1} is a condition on how fast $\om_j$ goes to $0$.  The
  existence of sequences of bound states satisfying Assumption \ref{T1} is
  guaranteed by Proposition \ref{th2.1}, where bound states with small
  frequencies are constructed as bifurcation from $0$ along a fixed
  radial bound state $Q$ of the equation $\Delta Q + |Q|^{\alpha_1} Q=Q$
  together with the estimate \eqref{exp.decay}. 
Note that the $\phi_j$ may be arbitrary excited states solutions
  of \eqref{soliton.eq}; in particular they may be sign-changing,
  non-radial, or complex-valued. Also note that we do not need the
  bound for $ \om^{-1/2}|\nabla \phi_\om(x)| $ in \eqref{exp.decay} for
  Theorems \ref{th1.1} and \ref{th1.3} below, but we assume it for all
  theorems for simplicity of presentation. For the same reason, we
  shall also set all initial positions $x_j$ to $0$.  Our
    assumption includes the finite multi-soliton case by setting
    $(\phi_j,\omega_j)=(0,0)$ for $j$ sufficiently large.

We have followed two independent approaches for the study of this problem, leading to two different types of results with different assumptions and conclusions. 
Before stating our main results, we need a preliminary lemma which will be proved in Section \ref{sec3}.
\begin{lemma}
\label{th3.1B}    
Let $d\ge 1$. For any $0<\alpha_1< \alpha_2 < \alpha_{\rm max}$
satisfying $ \frac {\alpha_2}{2+\alpha_2} \le \alpha_1
$, one can choose $r_0$ so that the following conditions hold.
\begin{gather}
\label{Ass3a}
  \max(1,\frac {d\alpha_1}2)<r_0 < 2+\alpha_1 ,
\\
\label{Ass3b}
\frac 12\le \frac{\alpha_1}{r_0}+\frac 1{r_2},
\\
\label{Ass3c}
 1<\frac{\alpha_1+1}{r_0}+\frac 1{r_2},
\end{gather}
where 
$r_2=2+\alpha_2$.
Furthermore, if $\alpha_1<4/d$, we can choose $r_0\le
2$.
\end{lemma}

\subsection{Infinite soliton trains}
We now state our two results on the existence of infinite soliton trains. The first approach of the first
theorem is based
on $L^p$-$L^q$ decay estimates for $e^{it \Delta}$.
The
Strichartz space $\stt$ will be defined in Section \ref{sec:perturbation}.

\begin{theorem}[Infinite train of solitons (i)]
\label{th1.1}  
Let $d \ge 1$ and assume Assumption \ref{F0} and
\begin{equation}
\label{th1.1-eq1}
 \frac {\alpha_2}{2+\alpha_2} \le \alpha_1.
\end{equation}
Let $r_2=2+\alpha_2$ and take any $r_0$ verifying \eqref{Ass3a}, \eqref{Ass3b}, and \eqref{Ass3c}.
Let $(\phi_j,\om_j)_{j \in
\N}$ be a sequence of bound states satisfying Assumption \ref{T1} with the
chosen $r_0$.  There exist constants $c_1>0$ and $v _\sharp \gg 1$
such that, for any infinite soliton train profile $R_\I$ given as in
\eqref{Rj.def} with parameters $v_j \in \R^d$, $x_j^0=0$, $\gamma_j\in
\R$ satisfying
\begin{equation}
\label{Ass2}
v_* = \inf_{j,k\in \N,  j\not = k} \, \sqrt {\om_j}\,|v_k-v_j| \ge v_\sharp,
\end{equation}
there exists a solution $u$ to \eqref{v100a0}  on $[0,\infty)$ satisfying
\begin{equation}
\label{th1.1-eq2}
\norm{(u-R_\I)(t)}_{ L^{r_2}} + \norm{u-R_\I}_{\stt} \le e ^{-c_1 v_* t},
\quad \forall t \ge 0.
\end{equation}
It is unique in the class of solutions satisfying the above estimate.
\end{theorem}

\begin{remark}[$L^2$-solutions]
By \eqref{th1.1-eq2} and H\"older
inequality,
\[
\norm{(u-R_\I)(t)}_{ L^{r}} \le e ^{-c_1 v_* t}, \quad \forall t \ge
0,\quad   \forall r\in[2, r_2].
\]
As we will show that $R_\I \in L^\I(0,\I; L^{r_0}\cap L^\I(\R^d))$ in
\eqref{eq3.2A}, we have $u \in L^\I(0,\I; L^{r_1}\cap L^\I(\R^d))$ where
$r_1=\max(2,r_0)$.  In the case $\alpha_1 < 4/d$, we can choose $r_0 \le
2$ by Lemma \ref{th3.1B}, and thus $u \in L^\I(0,\I; L^{2}(\R^d))$.
\end{remark}

\begin{remark}[Comparison to previous results]
Theorem \ref{th1.1} contains the pure power case $f(u)=|u|^\alpha u$
by writing $f(u)=|u|^\alpha u- 0 |u|^{\alpha+\epsilon} u$ for some
small $\epsilon>0$. It also
 includes the finite soliton train (multi-soliton) case by taking
$(\phi_j,\om_j)=(0,0)$ for $j$ sufficiently large.
In addition the range of exponents is larger
  than in \cite[Theorem 6.4]{LLT1}.  
 Hence Theorem \ref{th1.1} extends Theorems
1.1, 1.7, 6.3 and 6.4 in \cite{LLT1} in a unified approach (except that
\cite[Theorem 6.3]{LLT1} does not require \eqref{th1.1-eq1}). 
\end{remark}

\begin{remark}[$L^2$-subcritical nonlinearities]
If we use a pure Strichartz norm approach and do not use
  $L^{r_2}$ norm, we can construct infinite soliton trains for all
  $L^2$-subcritical or critical exponents $0 < \alpha_1<
  \alpha_2 \le 4/d$ as in \cite[Theorem 6.3]{LLT1}, without the
  restriction \eqref{th1.1-eq1}.
\end{remark}


In our second main result, we also control the train at the gradient level. The approach is based solely on Strichartz estimates.

\begin{theorem}[Infinite train of solitons (ii)]
\label{th1.2}  
Let $d \ge 1$ and assume Assumption \ref{F0} with $0<\alpha_1 < \frac
{4}{d+2}$. Let $(\phi_j,\om_j)_{j \in \N}$ be a sequence of bound
states satisfying Assumption \ref{T1} for some $r_0$. There exist constants $C>0$,
$c_1>0$, $c_2>0$, and $v _\sharp \gg 1$ such that, for any infinite
soliton train profile $R_\I$ given as in \eqref{Rj.def} with
parameters $v_j \in \R^d$, $x_j^0=0$, $\gamma_j\in \R$ satisfying
\begin{equation}
\label{Ass2b}
v_* : = \inf_{j,k\in \N,  j\not = k} \, \sqrt {\om_j}\,|v_k-v_j| \ge v_\sharp,
\end{equation}
and
\begin{equation}
\label{Ass4}
V_*:= \sum_{j\in \N} \bka{v_j} \omega_j ^{\frac {1}{\alpha_1} - \frac d{4}} <\I,
\end{equation}
there exists a unique solution $u$ to \eqref{v100a0} satisfying, for
some $T_0=T_0(V_*)\gg 1$,
\begin{equation}
e^{c_1 v_*t} \norm{u-R_\I}_{S([t,\I))} + e^{c_2 v_*t}
  \norm{\nabla(u-R_\I)}_{S([t,\I))} \le C, \quad \forall t \ge T_0.
\end{equation}
\end{theorem}

\begin{remark}[Examples of parameters choices]
Condition \eqref{Ass2b} requires sufficiently large relative speed,
while condition \eqref{Ass4} puts an upper bound on the growth of
$\bka{v_j}$. By  \eqref{Ass4} we may assume $r_0 \le 2$.
 One possible choice of parameters is
\begin{equation}
\label{eq1.16}
\om_j = 4^{-j}, \quad v_j = 2^{j+1} \bar v, \quad |\bar v|\gg 1.
\end{equation}
Condition \eqref{Ass4} can be satisfied ($V_* \lesssim \sum_j (4^{-j})^{-\frac 12+ \frac 1{\alpha_1}- \frac d4}<\I$) thanks to the assumption $\alpha_1 <
\frac {4}{d+2}$ (note this implies  $\alpha_1<1$ unless $d=1$).

In the above choice $V_*$ and $v_*$ grow linearly in $|\bar v|$.  In
the following choice $V_*= O( h(|\bar v|)|\bar v|)$ while $v_* = C
|\bar v|$ for any function $h>1$:
\begin{equation}
\label{eq1.17}
\om_j = 4^{-j}, \quad v_j = 
\begin{cases} 2^{j+1} h(|\bar v|)\bar v, & \mbox{if $j$ is odd} \\ - 2^{j+1} \bar v, & \mbox{if $j$ is even} \end{cases},
\quad
 |\bar v|\gg 1.
\end{equation}
\end{remark}

\begin{remark}[Infinite train starting at time $0$] We use large $T_0$ to off-set the contribution of large
$V_*$. If we impose that $V_*$ grows sub-exponentially in $v_*$, e.g.,
$V_* \le C (1+v_*)^M$ for some $M\ge 1$ (e.g.~$h(s)=(1+s)^{M-1}$ in \eqref{eq1.17}), we may take $T_0=0$ as in
\cite[Theorem 6.1]{LLT1}.
\end{remark}

\begin{remark}[Existence of infinite trains under \ref{F0} and \ref{T1}]
The proof of Theorem \ref{th1.1} uses a combination of $L^{r_2}$
  norm and Strichartz norm.
To estimate $|\eta|^{\alpha_1 +1}$
  in $L^{r_2}$ using $L^{r'}$-$L^r$ decay estimates, a
  restriction like \eqref{th1.1-eq1} is needed to avoid the limiting
  case $\alpha_1=0+$ and $\alpha_2=\alpha_{\rm max}-$. However, we claim
  that exponents excluded by \eqref{th1.1-eq1} are covered by Theorem
  \ref{th1.2} above. Indeed, let $\bar \alpha = \sup_{0< \alpha <
    \alpha_{\max}} \frac \alpha{2+\alpha}$. We have $\bar \alpha=1$
  for $d=1,2$ and $\bar \alpha = 2/d$ for $d \ge 3$. One then verifies
  that $\bar \alpha\le \frac 4{d+2}$ for all dimensions.

Hence
\emph{we can construct
  infinite soliton trains for all energy-subcritical nonlinearities
  satisfying} Assumptions \ref{F0} and \ref{T1}. 
\end{remark}

\begin{remark}[Comparison between Theorems \ref{th1.1} and \ref{th1.2}]
Theorem \ref{th1.1} applies for nonlinearities whose general form is not far from a power type nonlinearity, no matter what this power is ($\alpha_1$ can be any $H^1$-subcritical power).
Theorem \ref{th1.2} applies for nonlinearities that are sufficiently strong at $0$ ($\alpha_1$ has to be small), but with any kind of growth possible away from $0$.  For the choice of the profile, Theorem \ref{th1.1} is more flexible as it requires only some weak integrability condition \eqref{Ass1}, whereas Theorem \ref{th1.2} requires $L^2$-integrability of the profile (one take $r_0=2$ in \ref{T1}) \emph{and} its first derivative \eqref{Ass4}. 
\end{remark}

\subsection{Infinite kink-soliton trains}

In our next couple of theorems we let $d=1$ and consider in $\R$ a train of the form
\[
W= K+ R_\I
\]
where  $R_\I$ is as in \eqref{Rj.def}, and $K$ is a kink solution
of \eqref{v100a0} given by the same formula \eqref{soliton.formula}
but with the profile $\phi=\phi_K$ now being a half-kink satisfying
the same equation \eqref{soliton.eq} ($\phi''=\om \phi - f(\phi)$),
$0<\phi_K(s)<b$ for some $b>0$, and
\begin{equation}
\label{kink.profile}
\lim_{s \to- \infty}\phi_K(s) = b, \quad \phi_K'(s)<0 \quad \forall s\in \R, \quad
\phi_K'(0)=\min \phi_K',\quad
\lim_{s \to+ \infty}\phi_K(s) = 0.
\end{equation}
A
solution which converges to a profile $W$ as above at positive time
infinity will be called an {\it infinite kink-soliton train}. We are going to give two results of existence of infinite kink-soliton trains. Note that such object was never exhibited before, even in  integrable cases.

In addition to Assumption \ref{F0}, we make the following assumption, which in particular ensure the existence of a half-kink
satisfying \eqref{kink.profile} (see Proposition \ref{th:kink}).

\begin{assumptionF}\label{F1}
For some $\omega_0>0$, there is a first $b>0$ such that
for $h(s)=\omega_0 s - f(s)$,
\begin{equation}
\label{F1-1}
h(b)=0, \quad \int_0^b h(s)ds =0.
\end{equation}
Moreover, $h'(b)>0$, and for some $\ab\in [0,\alpha_2]$, 
\begin{equation}
\label{F1-2}
|f'(b+s)|+|s||f''(b+s)| \le C |s|^\ab + C |s|^{\alpha_2}, \quad \forall
s \in \R.
\end{equation}
\end{assumptionF}

Note that the nonlinearity \eqref{f-pq} admits a half-kink
when $d=1$. See Example \ref{ex4.2}.

We now state our second set of results on the existence of  infinite
kink-soliton trains.  Recall $\N_0=\{0\}\cup \N$.

\begin{theorem}[An infinite kink-soliton train (i)]
\label{th1.3}  
Let $d = 1$ and assume Assumptions \ref{F0}, \ref{F1} and 
\begin{equation}
\label{th1.3-eq1}
 \frac {\alpha_2}{2+\alpha_2} \le \alpha_1.
\end{equation}
Let
$r_2=2+\alpha_2$.  Then we can find $r_0$ satisfying
\eqref{Ass3a}--\eqref{Ass3c}. Assume that $\ab$ is such that
\begin{equation}
\label{Ass3d}
\frac 12 \le \frac {\ab }{r_0}  + \frac 1{r_2},\quad
 1 < \frac {\ab+1 }{r_0}  + \frac 1{r_2}.    
\end{equation}
Assume there is  a
sequence of bound states  $(\phi_j,\om_j)_{j \in \N}$ satisfying Assumption \ref{T1} with the chosen
$r_0$. Let $\phi_0=\phi_K$ be the kink profile to be given in Proposition 
\ref{th:kink}.  There exist constants $c_1>0$, and $v _\sharp
\gg 1$ such that, for the  infinite kink-soliton profile
$W=K+R_\I$, given as in \eqref{Rj.def}, with
any parameters $v_j \in \R$, $v_j < v_{j+1}$,
$x_j^0=0$, 
$\gamma_j\in \R$ for $j\in \N_0$ satisfying
\begin{equation*}
v_* = \inf_{j,k\in \N_0,  j\not = k} \, \sqrt {\om_j}\,|v_k-v_j| \ge v_\sharp,
\end{equation*}
there exists a unique solution $u$ to \eqref{v100a0} for $t \ge 0$ satisfying
\begin{equation}
\label{th1.3-eq3}
\norm{(u-W)(t)}_{ L^{r_2}} + \norm{u-W}_{\stt}\le e ^{-c_1 v_* t},
\quad \forall t \ge 0.
\end{equation}
\end{theorem}

%

\begin{theorem}[An infinite kink-soliton train (ii)]
\label{th1.4}  Let $d = 1$
 and assume Assumptions \ref{F0} and \ref{F1} with 
 $0<\alpha_1 < 4/3$. Let $(\phi_j,\om_j)$, $j \in
 \N$ be given and satisfying Assumption \ref{T1} for some $r_0$
which further satisfies
\begin{equation}
\label{th1.4-eq1}
  r_0(\alpha_1+1) < (\ab+1)(\alpha_1+2) . 
\end{equation}
Let $\phi_0=\phi_K$ be the kink
 profile to be given in Proposition \ref{th:kink}.  There
 exist constants $C>0$, $c_1>0$, $c_2>0$, $T_0\gg 1$ and $v _\sharp
 \gg 1$ such that, for the kink-soliton train profile $W=K+R_\I$ given
 as in \eqref{Rj.def} with any parameters $v_j \in \R$, $v_j > v_0$, $x_j^0=0$,
 $\gamma_j\in \R$ for $j\in \N_0$ and sufficiently large relative
 speed
\begin{gather}
\label{1.26}
v_* = \inf_{j\in \N,k\in \N_0,  j\not = k} \, \sqrt {\om_j}\,|v_k-v_j| \ge v_\sharp,
\\
\label{Ass4b}
V_*:= \sum_{j\in \N} \bka{v_j} \omega_j ^{\frac {1}{\alpha_1} - \frac d{4}} <\I,
\end{gather}
there exists a unique solution $u$ to \eqref{v100a0} for $t \ge T_0$ satisfying
\begin{equation}
e^{c_1 v_*t} \norm{u-W}_{S([t,\I))} + e^{c_2 v_*t}
  \norm{\nabla(u-W)}_{S([t,\I))} \le C, \quad \forall t \ge T_0.
\end{equation}
\end{theorem}

\begin{remark}
In Theorems \ref{th1.3} and \ref{th1.4}, the kink $K$ is on the left in the profile and its
velocity is less than the velocity of any soliton.  This picture can
be reversed by the symmetry $u(x,t)\to \tilde u(x,t)=u(-x,t)$.  
\end{remark}

\begin{remark}
In Theorem \ref{th1.4} we require upper bound $\alpha_1<4/3$ and lower bound \eqref{th1.4-eq1} on $\ab$. The bound \eqref{th1.4-eq1} is redundant if we choose a smaller $r_0$, e.g.~$r_0=1$, but is nontrivial if we take $r_0=2$.
\end{remark}

The rest of the paper is organized as follows: In Section \ref{sec2.1} we give an example of nonlinearity for which Assumption \ref{T1} is satisfied. In Section \ref{sec:perturbation} we give the general scheme of our proofs. 
In Section \ref{sec3} we prove
Theorems \ref{th1.1} and \ref{th1.2}. In Section \ref{sec5} we give Examples \ref{ex4.1} and \ref{ex4.2} for nonlinearities verifying Assumption \ref{F1} and we
  prove Theorems \ref{th1.3} and \ref{th1.4}.


\section{Existence of a family of bound states satisfying \ref{T1}}\label{sec2.1}

Assumption \ref{T1} is satisfied for the nonlinearity $f$ if, for example, $f$  satisfies Assumption \ref{F2} below. 

\begin{assumptionF}\label{F2}
 Suppose $f(u)=f_1(u)+f_2(u)$
where $f_1(u)=|u|^\alpha u$, $f_2(u)=g_2(|u|^2)u$, $g_2 \in
C^0([0,\infty),\R)\cap C^2((0,\infty),\R)$, $g_2(0)=0$ and
\begin{align*} 
|sg_2'(s)|+|s^2 g_2^{\prime\prime}(s)| \le
C_0(s^{\beta_1/2}+s^{\beta_2/2}), \quad \forall s >0,
\end{align*}
where $0<\alpha < \beta_1 \le \beta_2 < \alpha_{\rm max}$ and $C_0>0$.  
\end{assumptionF}

This assumption is more specific about the small $u$ behavior of
$f(u)$ than those in Assumption \ref{F0} so that we can have more control on the bound states with respect to their frequencies. In particular, we do not consider $f_1(u)$ with
opposite sign.

The following proposition gives an existence result of bound states with
small frequencies, obtained as the bifurcation from the radial ground
state $Q$ of the pure power nonlinearity, together with uniform
estimates.

\begin{proposition}[Bifurcation of solitons]\label{th2.1}
Let $d \ge 1$ and assume Assumption \ref{F2}. Let $Q(x)$ be the unique positive radial
solution of $\Delta Q + |Q|^{\alpha}Q = Q$ in $\R^d$.  There is a
small $\om_*=\om_*(d,\alpha,\beta_1,\beta_2,C_0)>0$ so that for all
$0<\om<\om_*$ there is a solution $\phi=\phi_\om$ of
\eqref{soliton.eq} of the form
\begin{equation}
\label{phi.ansatz}
\phi_\om (x)= \om^{1/\alpha}[Q(\om^{1/2}x) + \xi_\omega(\om^{1/2}x) ],  
\end{equation}
where $\norm{\xi_\omega}_{H^2} \le C \om^m$ with
$m=\frac{\beta_1/\alpha-1}{\min(1,\alpha)}>0$. Moreover, for any
$0<a<1$ there is a constant $D_a>0$ such that
\begin{equation}\label{exp.decay-2}
|\phi_\om(x)| + \om^{-1/2}|\nabla \phi_\om(x)| \le D_a \om^{1/\alpha}
e^{- a \om^{1/2} |x|} ,\quad \forall x\in \R^d,\ \forall
\om\in(0,\om_*).
\end{equation}
\end{proposition}

Note that we could allow $Q$ to be any radial excited state, provided we knew its non-degeneracy, i.e invertibility of $L_+$ in the proof below (such a result should be a consequence of the classifications results \cite{CoGaYa09,CoGaYa11}, however we did not pursue in that direction).

Before proving Proposition \ref{th2.1}, we recall  without proof the following classical lemma.

\begin{lemma}\label{th2}
Suppose $f(u)=g(|u|^2)u$, $g \in C^0([0,\infty),\R)$, $f(0)=0$ and 
\[
|sg'(s)|  \le C(s^{\alpha_1/2}+s^{\alpha_2/2}), \quad \forall s >0.
\]
For $W,\eta \in \C$ we have
\[
|f(W+\eta)-f(W)| \lesssim |\eta|(|W|^{\alpha_1}+|W|^{\alpha_2})+
|\eta|^{1+\alpha_1}+|\eta|^{1+\alpha_2}.
\]
\end{lemma}

\begin{proof}[Proof of Proposition \ref{th2.1}]
Since $Q$ is real and radial, we will look for real and radial $\xi_\omega$. For the sake of simplicity in notation, we drop the subscript $\omega$ during the proof.
Denoting $y=\om^{1/2}x$ and substituting \eqref{phi.ansatz} in
\eqref{soliton.eq}, we get
\[
(-\Delta_y +1)\xi = \om^{-\frac 1\alpha-1}f(\om^{1/\alpha}(Q+\xi)) -
|Q|^{\alpha}Q.
\]
It can be rewritten as
\begin{equation}
L_+ \xi = N(\xi) =N_1(\xi) + N_2(\xi),  
\end{equation}
where
\begin{align*}
L_+ &=  -\Delta_y +1 - (1+\alpha) |Q|^\alpha
\\
 N_1(\xi) &= f_1(Q+\xi) - f_1(Q) -  (1+\alpha) |Q|^\alpha\xi
\\
 N_2(\xi)  &=\om^{-\frac 1\alpha-1}f_2(\om^{1/\alpha}(Q+\xi)) .
\end{align*}
In the special case $f_2(u)=-|u|^\beta u$, we have $N_2(\xi)= 
-\om^{\frac {\beta} \alpha-1} |Q+\xi|^{\beta}(Q+\xi) $.

Let $X=H^2_{rad}(\R^d)$.  The properties of $L_+$ are well-known  (see e.g \cite{ChGuNaTs07}). It has one negative eigenvalue, its kernel in $L^2(\R^d)$ is spanned by $(\partial_{y_j}Q)_j$ and the rest of its spectrum is positive away from $0$. Hence for radial functions  $L_+:X \to L^2_{rad}$ is
invertible and we have
\[
C_3: = \norm{(L_+)^{-1}}_{\mathcal{B}(L^2_{rad}; X)}  < \infty.
\]
We have
\begin{equation}
|N_1(\xi)| \lesssim 1_{\alpha > 1} |Q|^{\alpha-1} |\xi|^2 + |\xi|^{1+\alpha}  
\end{equation}
\begin{equation}
|N_1(\xi_1)-N_1(\xi_2)| \lesssim 1_{\alpha > 1}
|Q|^{\alpha-1}(|\xi_1|+|\xi_2|)|\xi_1-\xi_2| +
(|\xi_1|+|\xi_2|)^{\alpha} |\xi_1-\xi_2|.
\end{equation}
We also have, by Assumption (F2) and Lemma \ref{th2},
\begin{equation}
|N_2(\xi)| \lesssim \om^{-\frac 1\alpha-1}
\tsum_{j=1}^2|\om^{1/\alpha}(Q+\xi)|^{1+\beta_j} =
\tsum_{j=1}^2 \om^{\frac {\beta_j} \alpha-1} |Q+\xi|^{1+\beta_j} .
\end{equation}
\begin{equation}
|N_2(\xi_1)-N_2(\xi_2)| \lesssim \tsum_{j=1}^2 \om^{\frac {\beta_j}
  \alpha-1} (|Q|+|\xi_1|+|\xi_2|)^{\beta_j} |\xi_1-\xi_2|.
\end{equation}
Denote $B_r=\{\xi \in X: \norm{\xi}_X \le r\}$ for $0<r<1$ and let $0<\omega<1$.  Because
$X$ is imbedded in $L^{2+2\alpha} \cap L^{2+2 \beta_2}$ for any
dimension $d$, we have, for some $C_4$,
\begin{equation}
\norm{N(\xi_1)-N(\xi_2) }_{L^2} \le
C_4\bke{(\norm{\xi_1}_{X}+\norm{\xi_2}_{X})^{\min(1,\alpha)} +
  \om^{\frac {\beta_1} \alpha-1} } \norm{\xi_1-\xi_2}_{X} ,
\end{equation}
for any $\xi_1,\xi_2 \in B_r$. Thus the map $\xi \mapsto
(L_+)^{-1}N(\xi)$ is a contraction map in $B_r\subset X$ for any $\om
\in (0,\om_*)$ if we choose $(2r)^{\min(1,\alpha)} = \om_*^{\frac
  {\beta_1} \alpha-1} < (4C_3C_4+1)^{-1}$.

Finally, standard argument for exponential decay  (see
\cite{Agmon} or \cite[Appendix]{GNT04}) shows that for any $a\in(0,1)$
\[
|\xi(x)|+|\nabla \xi(x)| \le o(1) e^{-a|x|}, \quad |Q(x)|+|\nabla Q(x)| \le
C e^{-a|x|},
\]
using the uniform bound $\norm{\xi}_{H^2} \ll 1$.  We get
\eqref{exp.decay-2} after rescaling.
\end{proof}

\section{The perturbation argument}\label{sec:perturbation}

We recall the definition of the Strichartz spaces  $\stt$ and $\ntt$ and the well known dispersive and Strichartz estimates. A pair of exponents $(q,r)$ is said to be \emph{(Schr\"odinger)-admissible} if 
\[
\frac{2}{q}+\frac{d}{r}=\frac{d}{2},\quad 2\leq q,r\leq+\infty,\quad (d,q,r)\neq (2,2,+\infty).
\]
Given a time $t\in\R$, the Strichartz space $\stt$ is defined via the norm
\[
\norm{u}_{\stt}=\sup_{\substack{(q,r) \text{ admissible} \\r \le r_{\rm Str} }}\norm{u}_{L_t^qL_x^r([t,+\infty)\times \R^d)}.
\]
Above $r_{\rm Str}= \I$ for $d\not =2$, but we choose $\alpha_{2}+2<r_{\rm Str}<\I$ when $d=2$ to stay away from the forbidden endpoint.
We denote the dual space by $\ntt=\stt^*$. Hence for any $(q,r)$ admissible, its norm verifies
\[
\norm{u}_{\ntt}\leq \norm{u}_{L_t^{q'}L_x^{r'}([t,+\infty)\times \R^d)}
\] 
where $q',r'$ are the conjugate exponents of $q$ and $r$. 

Let us recall the standard \emph{dispersive inequality}
\[
\norm{e^{it\Delta}u}_p \lesssim |t|^{-d\left(\frac12-\frac1p\right)}\norm{u}_{p'}\quad\text{for }t\neq 0, \;2\leq p\leq +\infty
\]
from which one can deduce the usual Strichartz estimate: 
\[
\norm{u}_{S([t_0,+\infty))}\lesssim \norm{u_0}_{L^2}+\norm{F}_{N([t_0,+\infty))}
\]
where for $u_0\in L^2(\R)$  $u$ solves on $[t_0,\infty)$  the following equation
\[
iu_t+\Delta u=F, \quad u(t_0)=u_0.
\]

For the proof of the main theorems with a profile $W=R_\I$ or
$W=K+R_\I$, we will consider the error term $\eta = u - W$, which
satisfies
\begin{equation}
\label{eta.eq}
i \partial_t \eta + \Delta \eta = - [f(W+\eta)-f(W)] - H, \quad H = f(W) -
\sum_{j \in \N_0} f(R_j).
\end{equation}
Above $R_0=0$ if $W=R_\I$ and $R_0=K$ if $W=K+R_\I$.
In Duhamel form,
\begin{equation}
\label{v600a2}
 \eta(t) = -i\int_t^\I e^{i(t-s)\Delta} [f(W+\eta)-f(W)+ H](s)\,ds.
\end{equation}

The proofs of Theorems \ref{th1.1} and \ref{th1.3} given in Sections \ref{sec3} and \ref{sec5} are self contained. For the proofs of Theorems \ref{th1.2} and \ref{th1.4}, we rely on the following generic result proved in \cite[Proposition 2.4]{LLT1}.

\begin{proposition}\label{prop_main_1}
Let $d\ge 1$ and assume Assumption \ref{F0}. 
Let $\sourceterm=\sourceterm(t,x):\,[0,\infty)\times \R^d \to \mathbb C$,
$\profile=\profile(t,x):\,[0,\infty)\times \R^d \to \mathbb C$
be given functions  which satisfy for some
 $C_1>0$, $C_2>0$,  $\lambda>0$, $T_0\geq 0$:
\begin{align} \label{v600a1}
&\|\profile(t) \|_{\infty}+    e^{\lambda t} \|\sourceterm(t) \|_2
 \le C_1, \qquad \forall\, t\ge T_0;\notag\\
&\|\nabla \profile(t) \|_2+ \| \nabla \profile(t) \|_{\infty} + e^{\lambda  t} \| \nabla \sourceterm(t)\|_2 \le C_2 ,
\qquad\;\forall\, t\ge T_0.
\end{align}
Consider the equation \eqref{v600a2}.
There exists a constant $\lambda_*=\lambda_*(d,\alpha_1, \alpha_2,
C_1)>0$ independent of $C_2$, and a time
$T_*=T_*(d,\alpha_1,\alpha_2,C_1,C_2)>0$ sufficiently large such that
if $\lambda\ge \lambda_*$ and $T_0\ge T_*$, then there exists a unique
solution $\eta$ to \eqref{v600a2} on $[T_0,+\infty)\times \mathbb R^d$
  satisfying
\begin{align}
e^{\lambda t} \| \eta\|_{S([t,\infty))}+ e^{\lambda c_1 t} \| \nabla \eta \|_{S([t,\infty))}  \le 1, \qquad \forall t\ge T_0.\label{v600a3}
\end{align}
Here $c_1>0$ is a constant depending only on $(\alpha_1,d)$.
\end{proposition}

\section{Construction of infinite soliton trains} \label{sec3}

\subsection{Proof of Theorem \ref{th1.1}}

In this section we prove Theorem \ref{th1.1} and construct infinite
soliton trains in $\R^d$, $d \ge 1$.  Note that \eqref{Ass1}
in Assumption \ref{T1} implies $A_2 := \sum_{j\in \N} \om_j^{\frac
  {1}{\alpha_1}}< \infty$, and
\begin{equation}
\label{eq3.2A}
\norm{R_\I(t)}_{L^\infty\cap L^{r_0}} \le \sum_{j\in \N}
\norm{R_j(t)}_{L^\infty\cap L^{r_0}} \lesssim \sum_{j\in \N}
(\om_j^{\frac {1}{\alpha_1}}+ \om_j^{\frac {1}{\alpha_1} - \frac d{2r_0}})= A_2 + A_1.  
\end{equation}
We first show the existence of the exponent $r_0$ and prove Lemma \ref{th3.1B}.

\begin{proof} [Proof of Lemma \ref{th3.1B}]
The idea is to choose $r_0=\max(1,\frac {d\alpha_1}2)+\e$ for some $0<\e
\ll 1$.  Clearly $r_0<2+\alpha_1$ for sufficiently small $\e>0$ since
$\alpha_1<\alpha_{\rm max}$. So \eqref{Ass3a} is satisfied.

In the case $\frac{d\alpha_1 }2\ge 1$, we claim
\[
\frac{\alpha_1}{\frac{d\alpha_1 }2}+ \frac 1{r_2}>\frac 12, \quad
\frac{\alpha_1+1}{\frac{d\alpha_1 }2}+ \frac 1{r_2}>1.
\]
Both are clear if $d\le 2$. For $d \ge 3$, both left sides become
strictly smaller if $\alpha_1$ is replaced by $\alpha_{\rm max} =
\frac{4}{d-2} $ and $r_2$ is replaced by $2+\alpha_{\rm max}$, but
are no less than the right sides by direct computation. Thus
\eqref{Ass3b} and \eqref{Ass3c} are satisfied for sufficiently small
$\e>0$.

In the case $\frac{d\alpha_1 }2< 1$, we claim
\[
\frac{\alpha_1}{1}+ \frac 1{r_2}>\frac 12, \quad
\frac{\alpha_1+1}{1}+ \frac 1{r_2}>1.
\]
The first inequality is a consequence of the assumption $\alpha_1\geq \alpha_2/(\alpha_2+2)$,
while the second is trivial.  Thus \eqref{Ass3b} and \eqref{Ass3c} are
satisfied for sufficiently small $\e>0$.

Suppose $\alpha_1<4/d$. In the case $\frac{d\alpha_1 }2\ge 1$, since
$\frac {d\alpha_1 }2<2$, $r_0=\frac {d\alpha_1 }2+\e < 2$ for sufficiently
small $\e>0$.  In the case $\frac{d\alpha_1 }2< 1$, $r_0=1+\e<2$.  The
proof of the lemma is complete.
\end{proof}

\begin{remark} Although we chose $r_0=\max(1,\frac
{d\alpha_1}2)+\e$ in the proof of Lemma \ref{th3.1B}, it is not
necessary for Theorem \ref{th1.1}. We only need $r_0$ to satisfy
\eqref{Ass3a}--\eqref{Ass3c}.
\end{remark}

We next estimate the source term in the equation for the error.

\begin{lemma}
\label{th3.2B}
Under the assumptions of Theorem \ref{th1.1}, the source term $H =
f(R_\I) - \sum_{j \in \N} f(R_j)$ satisfies, for some $c_1\in(0,a/2)$,
\[
\|H(\cdot,t)\|_{L^{\infty} \cap L^{r_2'}} \le C e^{-c_1v_*t}.
\]
\end{lemma}

\begin{proof} 
Fix $t>0$. For any $x \in \R^d$, choose $m=m(x)\in\N$ so that $\phi_m$ is a
nearest soliton, i.e.
\[
|x-v_m t| = \min_{j \in \N} |x-v_j t|.
\]
For $j \not = m$, we have
\begin{equation}\label{eq:v-diff}
|x-v_j t| \ge \frac 12 |v_j t-v_mt| = \frac t2 |v_j -v_m|.
\end{equation}
Thus, by \eqref{exp.decay}, we have
\begin{equation}\label{eq3.3A}
|(R_\I-R_m)(x,t)|\le \sum_{j \not = m}|R_j(x,t)| \le \delta_m(x,t)
:= \sum_{j \not = m} D_a \om_j^{\frac {1}{\alpha_1}}e^{-a \om_j^{1/2}|x-v_jt|}.
\end{equation} 
Hence, by  \eqref{Ass1}, the definition of $v_*$ \eqref{Ass2} and \eqref{eq:v-diff}, we have
\begin{equation}\label{eq3.4A}
\delta_m(x,t) \le \sum_{j \not = m} D_a \om_j^{\frac {1}{\alpha_1}} e^{-\frac 12
  a v_* t} = D_a A_2 e^{-\frac 12 a v_* t}.
\end{equation}
Denote $A_3 =
\sup_{0<s<\norm{R_\I}_{L^\I}} |f'(s)|$. By Lemma \ref{th2} and \eqref{eq3.2A},
we have
\begin{align*}
|H(t,x)| &\le |f(R_\I) - f(R_m)| + \sum_{j \not = m} |f(R_j)| 
\\
&\le A_3 |R_\I-R_m| +  {\tsum_{j \not = m}} A_3|R_j| \le 
2A_3  {\tsum_{j \not = m}} |R_j| 
 \le 2A_3 \delta_m(t,x).
\end{align*}
In particular,
\begin{equation}
\label{eq3.3}
\|H(t)\|_{ L^\infty} \le 2 D_a A_2A_3  e^{-\frac 12 a v_* t}.
\end{equation}
Condition \eqref{Ass3c} is equivalent to $\frac 1{r_2'}<\frac
{1+\alpha_1}{r_0}$.  Thus we can choose $s$ so that
\begin{equation}
\label{eq3.4}
\frac {1+\alpha_1}{r_0}>\frac 1s>\frac 1{r_2'}, \quad s > 1.
\end{equation}
The first inequality of \eqref{eq3.4}
ensures that  
\[
\frac {\alpha_1+1}{\alpha_1} - \frac d{2s} > \frac {1}{\alpha_1} - \frac
d{2r_0},
\]
and hence, using  \eqref{Ass1},
\begin{align*}
\sum_{j\in \N}\norm{f(R_j)}_{L^s }& \lesssim
\sum_{j\in \N} \norm{|R_j|^{\alpha_1+1}+ |R_j|^{\alpha_2+1}}_{L^s}
\lesssim \sum_{j\in \N}  \om_j^{\frac {\alpha_1+1}{\alpha_1} - \frac d{2s} } < C<\infty.  
\end{align*}
Since $r_0< s(1+\alpha_1)<s(1+\alpha_2)<\infty$ by \eqref{eq3.4}, we have
by \eqref{eq3.2A}
\[
\norm{f(R_\I)}_{L^s} \lesssim \norm{R_\I}^{1+\alpha_1}_{L^\infty \cap L^{r_0}}+
\norm{R_\I}^{1+\alpha_2}_{L^\infty \cap L^{r_0}} < C <\infty.
\]
Thus
\begin{equation}
\label{eq3.5}  
\|H(t)\|_{L^{s}}  < \norm{f(R_\I)}_{L^s}+ \sum_{j\in \N}\norm{f(R_j)}_{L^s }  <C< \infty.
\end{equation}
By H\"older inequality between $L^\infty$ and $L^s$ using \eqref{eq3.3} and \eqref{eq3.5}, we have
\[
\|H(t)\|_{L^{r}} \le C e^{-(1-s/r)\frac a2 v_* t}, \quad \forall r \in (s,\I).
\]
Since $s<r_2'<\I$ by \eqref{eq3.4}, we get the desired conclusion.
\end{proof}

We now prove Theorem \ref{th1.1}.

\begin{proof}[Proof of Theorem \ref{th1.1}]
The existence of $r_0$ has been shown in Lemma \ref{th3.1B}. We now
fix such a choice.  The difference $\eta = u - R_\I$ satisfies
equation \eqref{v600a2} with $W=R_\I$ and $H=f(R_\I)-\sum_{j\in
  \N}f(R_j)$. Denote the right side of \eqref{v600a2} as $\Phi \eta$.
We will show it is a contraction mapping and has a unique fixed point
$\eta=\Phi \eta$ in the class
\begin{equation}
\label{eq3.6}  
\norm{\eta(t)}_{L^{r_2}} +\norm{\eta}_{\stt}\le e^{-c_1 v_* t}, \quad \forall t\ge 0.
\end{equation}
We first show boundedness and suppose $\eta$ satisfies
\eqref{eq3.6}. By H\"older inequality,
\[
\norm{\eta(t)}_{ L^{r}} \le e ^{-c_1 v_* t}, \quad \forall t \ge
0,\quad  \forall r\in [2,r_2].
\]
We have
\[
\norm{\Phi \eta(t)}_{L^{r_2}} \le C\int_t^\infty |t-\tau|^{-\theta }
\bke{\norm{ f(W+\eta) -f(W)}_{L^{r_2'}} +
  \norm{H(\tau)}_{L^{r_2'}}}d\tau,
\]
where $\theta = d(\frac 12 - \frac 1{r_2})$, and $0<\theta<1$ since
$2<r_2<2+\alpha_{\rm max}$.

By Lemma \ref{th3.2B} we have $\norm{ H(\tau)}_{L^{r_2'}} \le C
e^{-c_1 v_* \tau}$.  By Lemma \ref{th2},
\begin{equation}
\label{eq3.7}
\norm{ f(W+\eta) -f(W)}_{L^{r_2'}} \lec \norm{|\eta| (|W|^{\alpha_1} +
  |W|^{\alpha_2})}_{L^{r_2'}} + \norm{|\eta|^{\alpha_1+1}+
  |\eta|^{\alpha_2+1}}_{L^{r_2'}}.  
\end{equation}
The first term on the right side is bounded by H\"older inequality
\[
\norm{|\eta| (|W|^{\alpha_1} + |W|^{\alpha_2})}_{L^{r_2'}} \le
 (1+\norm{W}_{ L^\infty}^{\alpha_2-\alpha_1}) \norm{W}_{L^{r_0 }\cap
     L^\infty}^{\alpha_1} \norm{\eta}_{L^{2} \cap L^{r_2}}\le C e^{-c_1v_*t}
\]
if
\[
\frac {\alpha_1}{\infty}+\frac 1{r_2} \le \frac 1{r_2'} \le \frac
      {\alpha_1}{ r_0}+\frac 1{2}.
\]
The first inequality is always true since $r_2'\le 2 \le r_2 $. The
second inequality is correct if \eqref{Ass3b} holds. Thus this term
can be estimated.

The last term of \eqref{eq3.7} is bounded by
\[
 \norm{|\eta|^{\alpha_1+1}+ |\eta|^{\alpha_2+1}}_{L^{r_2'}}\lec
 \norm{\eta}^{\alpha_1+1}_{L^{r_2'(\alpha_1+1)}}
 +\norm{\eta}^{\alpha_2+1}_{L^{r_2'(\alpha_2+1)}},
\]
which is bounded by  $C e^{-c_1v_*t}$ since
\[
2 \le r_2'(\alpha_1+1) < r_2'(\alpha_2+1) \le r_2,
\]
due to \eqref{th1.1-eq1} and $r_2=2+\alpha_2$.

Combining the above we have, assuming \eqref{eq3.6},
\begin{align*}
\norm{\Phi \eta(t)}_{L^{r_2}} &\le 
\int_t^\infty |t-\tau|^{-\theta } Ce^{-c_1 v_* \tau}d\tau 
\le  C v_*^{-1+\theta } e^{-c_1 v_* t}
\end{align*}
for all $t\ge 0$, which is bounded by $\frac 14 e^{-c_1 v_* t}$ if
$v_*$ is sufficiently large.

For the Strichartz estimate, since $(2/\theta,r_2)$ is admissible, we
have with $a=(2/\theta)'$
\begin{align*}
\norm{\Phi \eta}_{\stt} &
\lec \norm{f(W+\eta)-f(W)+H}_{L^a(t,\I;L^{r_2'})} 
\\
&\lec \norm{e^{-c_1 v_* \tau}}_{L^a(t,\I)}  \lec v_*^{-1/a}e^{-c_1 v_* t},
\end{align*}
for all $t\ge 0$, which is bounded by $\frac 14 e^{-c_1 v_* t}$ if
$v_*$ is sufficiently large.

Consider now the difference estimate. Suppose both $\eta_1$ and
$\eta_2$ satisfy \eqref{eq3.6}. Denote $\eta=\eta_1-\eta_2$ and
\[
\delta = \sup_{t>0} e^{c_1 v_* t}\bke{ \norm{\eta(t)}_{ L^{r_2}}+
  \norm{\eta}_{ \stt}} \le 2.
\] 
We have
\[
\norm{(\Phi \eta_1 - \Phi \eta_2)(t)}_{L^{r_2}} \le C\int_t^\infty
|t-\tau|^{-\theta } \norm{f(W+\eta_1) -f(W+\eta_2)}_{L^{r_2'}} (\tau)
\,  d\tau.
\]
By Lemma \ref{th2} again with $W$ replaced by $W+\eta_2$,
\begin{align}
\norm{ f(W+\eta_1) -f(W+\eta_2)}_{L^{r_2'}} &\lec \norm{|\eta| (|W+\eta_2|^{\alpha_1} \!+
  |W+\eta_2|^{\alpha_2})}_{L^{r_2'}} + \norm{| \eta|^{\alpha_1+1}+
  | \eta|^{\alpha_2+1}}_{L^{r_2'}}  
\nonumber
\\
&\lec  \norm{|\eta| (|W|^{\alpha_1} +
  |W|^{\alpha_2})}_{L^{r_2'}} + \norm{| \eta| (E^{\alpha_1}+
  E^{\alpha_2})}_{L^{r_2'}}
\label{eq3.8}
\end{align}
where $E=|\eta_1|+|\eta_2|$.  The first term is already bounded above
\[
\norm{|\eta| (|W|^{\alpha_1} + |W|^{\alpha_2})}_{L^{r_2'}} \le C
\norm{\eta}_{L^{2} \cap L^{r_2}} \le C \delta e^{-c_1 v_* t}.
\]
The last term of \eqref{eq3.8} is bounded similarly as above
\[
\norm{| \eta| (E^{\alpha_1}+ E^{\alpha_2})}_{L^{r_2'}}\le
\norm{\eta}_{L^{2} \cap L^{r_2}} (\norm{E}^{\alpha_1}_{L^{2} \cap
  L^{r_2}} +\norm{E}^{\alpha_2}_{L^{2} \cap L^{r_2}}) \le C \delta
e^{-c_1 (1+\alpha_1)v_*t}.
\]
Thus
\begin{align*}
\norm{(\Phi \eta_1 - \Phi \eta_2)(t)}_{L^{r_2}} &\le 
\int_t^\infty |t-\tau|^{-\theta } C\delta e^{-c_1 v_* \tau}d\tau 
\\
&\le  C \delta v_*^{-1+\theta } e^{-c_1 v_* t}
\end{align*}
for all $t\ge 0$, which is bounded by $\frac 14 \delta e^{-c_1 v_* t}$
if $v_*$ is sufficiently large. 

We also have (recall $a=(2/\theta_1)'$)
\begin{align*}
\norm{\Phi \eta_1 - \Phi \eta_2}_{\stt} &
\lec \norm{f(W+\eta_1) -f(W+\eta_2)}_{L^a(t,\I;L^{r_2'})} 
\\
&\lec \norm{\delta e^{-c_1 v_* \tau}}_{L^a(t,\I)}  \lec \delta v_*^{-1/a}e^{-c_1 v_*t} ,
\end{align*}
for all $t\ge 0$, which is bounded by $\frac 14 \delta e^{-c_1 v_* t}$
if $v_*$ is sufficiently large. 

We have shown that $\Phi$ is a contraction mapping and hence has a
unique fixed point in the set \eqref{eq3.6}.  The proof of Theorem
\ref{th1.1} is complete.
\end{proof}

\begin{remark}
 The assumption \eqref{th1.1-eq1}
  is used to estimate $L^{r_2'}$. To estimate $|\eta|^{\alpha_1 +1}$ in
  $L^{r_2}$ using $L^{r'}$-$L^r$ decay estimates, a restriction like
  \eqref{th1.1-eq1} is needed to avoid the limiting case $\alpha_1=0+$
  and $\alpha_2=\alpha_{\rm max}-$.

The condition \eqref{Ass3b} is used to bound the linear term in
  $\eta$, while \eqref{Ass3c} is used to bound the source term (it
  ensures the existence of $s$ in the proof of Lemma \ref{th3.2B}).

In \eqref{Ass3a}, we need $r_0 \ge 1$ for \eqref{eq3.2A}.  We
  need $r_0 > \frac {d\alpha_1}2$ so that the exponent in \eqref{Ass1}
  is positive.  The condition $r_0<\alpha_1+2$ in \eqref{Ass3a} is
  redundant and follows from \eqref{Ass3c}.
\end{remark}

%

\subsection{Proof of Theorem \ref{th1.2}}
\label{sec3.2}

In this section we prove Theorem \ref{th1.2} and construct infinite
soliton trains in $\R^d$, $d \ge 1$. All along this section, we assume that we are under the assumptions of Theorem \ref{th1.2}, in particular
we suppose that we are given a sequence of bound states $(\phi_j, \om_j)$ for $j\in \N$ satisfying assumptions \ref{T1}, \eqref{Ass2b} (with $v_\sharp$ to be determined later) and \eqref{Ass4}.

We first prove the following lemma.

\begin{lemma}
\label{th3.1}
Let $a\in(0,1)$ be given by Assumption \ref{T1}. For $\lambda=a\min(1,2a)v_*/4>0$, we have 
\begin{align} 
&\|R_\I(t) \|_{\infty}+    e^{\lambda t} \|\sourceterm(t) \|_2
 \le C, \qquad \forall\, t\ge 0;\notag\\
&\|\nabla R_\I(t) \|_2+ \| \nabla R_\I(t) \|_{\infty} + e^{\lambda  t} \| \nabla \sourceterm(t)\|_2 \le C(1+V_*),
\qquad\;\forall\, t\ge 0.
\end{align}
where $H$ is the source term defined by $H = f(R_\I) - \sum_{j \in \N} f(R_j)$.
\end{lemma}

\begin{proof}
Equation \eqref{Ass1} in Assumption \ref{T1} implies $A_2 := \sum_{j\in
  \N} \om_j^{\frac {1}{\alpha_1}}< \infty$, and
\[
\norm{R_\I(t)}_{L^\infty\cap L^{r_0}} \le \sum_{j\in \N}
\norm{R_j(t)}_{L^\infty\cap L^{r_0}} \lesssim \sum_{j\in \N}
(\om_j^{\frac {1}{\alpha_1}}+ \om_j^{\frac {1}{\alpha_1} - \frac d{2r_0}})= A_2 + A_1.
\]
We also have for $1 \le r \le \I$
\begin{equation}\label{eq:411}
\norm{\nabla R_\I(t)}_{L^r } \lesssim \sum_{j\in \N} \norm{\nabla
  R_j(t)}_{L^r } \lesssim \sum_{j\in \N} \omega_j ^{\frac {1}{\alpha_1} +
  \frac 12 - \frac d{2r}} + \sum_{j\in \N}|v_j| \omega_j ^{\frac{1}{\alpha_1} - \frac d{2r}}.
\end{equation}
If we take $r=2$, we have $\frac {1}{\alpha_1} + \frac 12 - \frac d{2r} \ge
\frac {1}{\alpha_1} - \frac d{2r_0}$ for all dimensions since $r_0 <
2+\alpha_{\rm max}$. Thus the first sum of the right hand side of \eqref{eq:411} is finite for $r\in[2,\infty]$
by
\eqref{Ass1}.  The second sum is also finite for $r\in[2,\infty]$ by \eqref{Ass4}. 
Thus
\[
\norm{\nabla R_\I(t)}_{L^2\cap L^\I} \lec A_1+V_*.
\]

We next consider the estimates of $H=f(R_\I)-\sum_{j\in \N} f(R_j)$.
Fix $t>0$. As in the proof of Lemma \ref{th3.2B}, take any $x \in \R^d$ and choose $m=m(x)\in \N$ so that
$\phi_m$ is a nearest soliton, i.e.
\[
|x-v_m t| = \min_{j \in \N} |x-v_j t|.
\]
Since $\alpha_1< \alpha_{\rm max}$ and $r_0 < 2+\alpha_1$, there exists $s=
\frac {\alpha_1+2-\epsilon}{\alpha_1+1}$ with  $0<\epsilon \ll 1$ such
that
\[
r_0 < 2+\alpha_1 -\epsilon,\quad
\frac {\alpha_1+1}{\alpha_1} - \frac d2 \cdot \frac 1s \ge \frac {1}{\alpha_1} -
\frac d{2r_0}.
\]
From arguments identical to those of the proof of Lemma \ref{th3.2B}, we have 
\[
\|H(t)\|_{L^{r}} \le C e^{-c(1-s/r) v_* t}, \quad \forall r \in (s,\I),
\]
with acceptable $r$ including $
\frac {\alpha_1+2}{\alpha_1+1}$ and $2$.

To estimate $\norm{\nabla H(t)}_{L^2}$, recall that by the Chain Rule we have
\begin{equation}\label{4.5}
\begin{split}
\nabla H&=\nabla( f(R_\I) ) -\sum_{j\in \N} \nabla (f(R_j) )  \\
&=  \sum_{j\in \N} ( f_z(R_\I)-f_z(R_j) ) \nabla R_j 
+ \sum_{j\in \N} ( f_{\bar z}(R_\I)-f_{\bar z} (R_j) ) \overline{\nabla R_j}.
\end{split}  
\end{equation}
Here, we denoted $f_z=\frac{\partial}{\partial z}f$ and $f_{\bar z}=\frac{\partial}{\partial \bar z}f$ the Wirtinger derivatives of $f$. 
Thus (here $x$ and  $m=m(x)$ are still as above), we have
\begin{align*}
|\nabla H(t,x)| &\lesssim \sum_{j \not =m} |\nabla R_j| +
\big(|f_z(R_\I)-f_z(R_m)|+|f_{\bar z}(R_\I)-f_{\bar z}(R_m)|\big)\,|\nabla R_m|
\\ &\lesssim  \sum_{j \not = m}\bka{v_j} \om_j^{1/\alpha_1}e^{-a \om_j^{1/2}|x-v_jt|}
+ (\delta_m(t,x))^{\min(1,\alpha_1)} \bka{v_m}  \om_m^{1/\alpha_1}e^{-a \om_m^{1/2}|x-v_mt|}
\\ &\lesssim  \sum_{j \not = m}\bka{v_j} \om_j^{1/\alpha_1}e^{-\frac 12 a \om_j^{1/2}|x-v_jt|}
e^{-\frac a4 v_*t}
+   e^{-\frac a2\min(1,\alpha_1) v_*t}  \bka{v_m}  \om_m^{1/\alpha_1}e^{-a \om_m^{1/2}|x-v_mt|}
\\ &\lesssim  e^{-\lambda t}  
\sum_{j \in \N}\bka{v_j} \om_j^{1/\alpha_1}e^{-\frac 12 a \om_j^{1/2}|x-v_jt|}
\end{align*}
where $\delta_m(t,x)$ is defined and estimated in \eqref{eq3.3A}--\eqref{eq3.4A}, and
$\lambda = \frac a4\min(1,2\alpha_1) v_*$. Thus
\[
\norm{\nabla H(t)}_{L^2} \lec  e^{-\lambda t}  
\sum_{j \in \N} \bka{v_j} \om_j^{1/\alpha_1-\frac d{4}} \lec V_*  
\]
by Assumption \eqref{Ass4}. The proof of Lemma \ref{th3.1} is complete.
\end{proof}

We now prove Theorem \ref{th1.2}.

\begin{proof}[Proof of  Theorem \ref{th1.2}]
By Lemma \ref{th3.1}, there exists $v_\sharp$ such that if $v_*>v_\sharp$, then the hypothesis \eqref{v600a1} of Proposition
\ref{prop_main_1} is satisfied under the assumptions of Theorem
\ref{th1.2}, with $W=R_\I$ and $H= f(R_\I)-\sum_{j\in \N}f(R_j)$.  By
Proposition \ref{prop_main_1}, there exist $T_0$ large enough and $\eta \in
C([T_0,\infty),H^1)$ with $\| \langle \nabla \rangle
  \eta\|_{S([t,\infty))}$ (in particular $\|\eta(t)\|_{H^1}$) decaying
 exponentially in $t$.
\end{proof}

\begin{remark}
One may tend to relax the exponent $2$ in the norm $\norm{\nabla W}_{L^2 }$ so that $\nabla W$ is not
that localized. However, $\norm{\nabla W}_{L^{2+\beta_1} }$ with
$\beta_1 < 0.01$ is used in the proof of Proposition \ref{prop_main_1}. It
would not gain much trying to optimize it.
\end{remark}

\section{Construction of infinite kink-soliton trains} \label{sec5}

In this section we prove Theorems \ref{th1.3} and \ref{th1.4}, and construct a train made of
infinitely many solitons and a half-kink for space dimension 1.

We first examine Assumption \ref{F1} and give some examples. Estimate
\eqref{F1-2} is natural since $f'$ is H\"older continuous.  
If
$f'(b)\not =0$, we can only take $\ab=0$.  Otherwise, we may take
$\ab=1$ if $f$ is locally $C^{1,1}$ near $b$.
For certain $f(s)$ we have $\ab>1$.

\begin{example} \label{ex4.1}
Let $f(s) = s- \frac{\sin |s|}{|s|}s $.  If we write $f(s)=f_1(s)+f_2(s)$ with
$f_1(s)=\frac 13 |s|^2s$ and $f_2(s) = s - \frac {\sin |s|}{|s|}s -
\frac 13 |s|^2s = O(s^5)$, $f$ satisfies Assumptions \ref{F0} and \ref{F2} with  
$\alpha_1=2$ and $\alpha_2=4$. We can choose $r_0 = 1+\epsilon$, $0<\epsilon\ll 1$, for Assumption \ref{T1}. The function $f(s)$
also satisfies
Assumption \ref{F1} with $\omega=1$,
$b=2\pi$, $h(s)=\sin s$ and $h'(b)=1$. Moreover,
\[
f(2\pi)=2\pi\not=0,\quad
|f'(2\pi+s)| =  |1-\cos s | \le C s^\ab, \quad \ab=2.
\]
 Hence conditions \eqref{th1.3-eq1}-\eqref{Ass3d} are satisfied. Thus we can construct
infinite kink-soliton trains using Theorem \ref{th1.3}. Since $\alpha_1>4/3$, Theorem \ref{th1.4} does not apply to this example. \hfill \qed
\end{example}

\begin{example}\label{ex4.2}  let $f(s) = |s|^{\alpha}s - |s|^{\beta}s$,
$0<\alpha<\beta<\infty$. Clearly $f$ satisfies Assumptions \ref{F0} and \ref{F2}
with $\alpha_1=\alpha$ and $\alpha_2=\beta$.
The conditions $h(b)=0=\int_0^b h(s)ds$ in Assumption \ref{F1}  give
\[
\om= b^\alpha - b^{\beta} = \frac 2{2+\alpha} b^\alpha - \frac
2{2+\beta} b^{\beta}.
\]
Thus
\[
b^{\beta-\alpha} = \frac{\alpha(2+\beta)}{(2+\alpha)\beta} \in \Big(
\frac \alpha\beta,1\Big), \quad
\om= b^\alpha (1-b^{\beta-\alpha})>0,
\]
and
\[
h'(b) = \omega - (1+\alpha) b^\alpha+ (1+\beta) b^{\beta} = 
- \alpha b^\alpha+ \beta b^{\beta}>0.
\]
Thus \eqref{F1-1} can be always satisfied by unique $\omega>0$ and
$b>0$. For \eqref{F1-2}, we have $\ab=0$ for most pair
$(\alpha,\beta)$.  Theorem \ref{th1.3}
is not applicable in those cases. The exception is when $0=f'(b)=(1+\alpha) b^\alpha-
(1+\beta) b^{\beta}$, hence $b^{\beta-\alpha} =
\frac{\alpha(2+\beta)}{(2+\alpha)\beta} = \frac {1+\alpha}{1+\beta}$,
or $\alpha \beta=2$.  Thus the exceptional case is
\begin{equation*}
\ab=1 \quad \text{if}\quad 0<\alpha<\sqrt 2, \quad \beta
= \frac 2\alpha.  
\end{equation*}
Since $d\alpha/2<1$, we can take $r_0=1$.
 Conditions \eqref{th1.3-eq1}-\eqref{Ass3d} imply
\begin{equation}
\label{eq5.10}
\frac{\sqrt 5-1}2 \le \alpha < \sqrt 2, \quad \beta
= \frac 2\alpha.  
\end{equation}
Thus for $\alpha$ satisfying \eqref{eq5.10}, using Theorem \ref{th1.3} we can construct infinite kink-soliton trains for the nonlinearity
$f(u)=(|u|^\alpha-|u|^{2/\alpha})u$. 
On the other hand, by Theorem \ref{th1.4} we can construct infinite
kink-soliton trains if 
\begin{equation}
0<\alpha<4/3, \quad \alpha< \beta < \infty.
\end{equation}
We do not need $\alpha \beta=2$.
Indeed, since $d\alpha/2<1$ for $\alpha<4/3$, we can take $r_0=1$, and condition
\eqref{th1.4-eq1} is satisfied for 
any $\ab\ge 0$. 
We can choose $\omega_j$ and $v_j$ as in \eqref{eq1.16} or \eqref{eq1.17}.
In comparison, Theorem \ref{th1.4} covers more exponents than Theorem \ref{th1.3} except when
$4/3\le \alpha < \sqrt 2$ and $\beta=2/\alpha$.
\hfill \qed
\end{example}


The existence of half-kink profiles is guaranteed by the following result.

\begin{proposition}\label{th:kink} 
Let $d=1$ and assume Assumptions \ref{F0} and \ref{F1}.  There is a solution
$\phi_K(s)$ of
\begin{equation*}
\phi_K'' = \omega_0 \phi_K - f(\phi_K)
\end{equation*}
such that $0<\phi_K(s)<b$,
\begin{equation*}
\lim_{s \to- \infty}\phi_K(s) = b, \quad \phi_K'(s)<0 \quad \forall s\in \R, \quad
\phi_K'(0)=\min \phi_K',\quad
\lim_{s \to+ \infty}\phi_K(s) = 0,
\end{equation*}
and that, for any $0<a<\min(\om_0,h'(b))$, there is $ D_a>0$ so that
\begin{equation*}
\mathbf{ 1}_{s<0}(b-\phi_K(s))+ \mathbf{ 1}_{s\ge 0}\, \phi_K(s) +
|\phi_K'(s)| \le D_a e^{-a|s|},\quad \forall s \in \R.
\end{equation*}
\end{proposition}

Proposition \ref{th:kink} can be easily proved using classical ordinary differential equations techniques (see e.g.  \cite[Proposition 1.12]{LLT1}).
As mentioned in Section \ref{sec1}, a kink solution of \eqref{v100a0} with
parameters $(v_0,\gamma)$ is (setting the spatial translation to $x_0=0$)
\[
K(t,x) = \phi_K(x-v_0t) e^{i (\omega_0 t + \frac 12 v_0x -
  \frac 14 v_0^2t + \gamma )}.
\]
For notational simplicity, we denote $ K=R_0$ and we consider the kink-soliton train profile
\[
W = K + R_\infty = \sum_{j=0}^{\I} R_j
\]
where $R_\I$ and $R_j$, $j>0$, are given in \eqref{Rj.def}.

\subsection{Proof of Theorem \ref{th1.3}} 

We will
solve the difference $\eta = u - W$ in the class \eqref{th1.3-eq3}.

To start the proof, we note that, because $\ab$ satisfies the same
conditions as $\alpha_1$, we can choose $r_0$ as in Lemma \ref{th3.1B}
to satisfy \eqref{Ass3d} in addition to \eqref{Ass3a}--\eqref{Ass3c}.
From now on we fix $r_0$.

We start by estimating the source term.

\begin{lemma}
\label{th5.2}
Under the assumptions of Theorem \ref{th1.3}, the source term $H =
f(W) - \sum_{j=0}^{\I} f(R_j)$ satisfies, for some $c_1>0$,
\[
\|H(\cdot,t)\|_{L^{\infty} \cap L^{r_2'}} \le C e^{-c_1v_*t}.
\] 
\end{lemma}

\begin{proof}
By \eqref{Ass3d}, we have $\frac {\ab+1}{r_0} > \frac
1{r_2'}$. We can choose $s$ as in the proof of Lemma \ref{th3.2B} to
satisfy \eqref{eq3.4} and $\frac {\ab+1}{r_0} > \frac
1s>\frac 1{r_2'}$.

For $x \ge \frac 12(v_0+v_1)t$, the contribution from $R_0$ is the
same as if $R_0$ were a soliton. Thus the estimate follows from Lemma
\ref{th3.2B}.

For $x \le \frac 12(v_0+v_1)t$, we have $H = (f(W)-f(K)) -
\sum_{j=1}^\I f(R_j)$. In the proof of Lemma \ref{th3.2B} we have
shown
\begin{gather}
\label{eq5.5}
|R_\infty(t,x)| \le C e^{-\frac 12 a v_* t}, \quad \norm{R_\I}_{L^{r_0}} \le C,  
\\
\Big|\sum_{j=1}^\I f(R_j)(t,x)\Big|\le Ce^{-\frac 12 a v_*t}, \quad
\norm[\Big]{\sum_{j=1}^\I f(R_j)}_{L^s} \le C.\notag
\end{gather}
For simplicity in notations, we assume now that $v_0=\gamma_0=0$. This causes no loss of generality since \eqref{v100a0} is invariant under a Galilean transform and it guarantees that the left part of the kink is approximately $b$ without correction by a phase factor containing $e^{i\frac12v_0x}$.
By Assumption \ref{F1}, the mean value theorem, and since $K,R_\I\in L^\I$, we have
\[
|f(W)-f(K)| =|f(b+K-b+R_\I)-f(b+K-b)| 
\lesssim (||K|-b|+|R_\infty|)^{\ab} |R_\infty|.
\] 
We first derive
\[
|f(W)-f(K)| \le Ce^{-\frac 12 a v_* t}.
\] 
Because $r_0< (1+\ab)s$,
\begin{align*}
\norm{f(W)-f(K)}_{L^s} &\le \norm{||K|-b|^{\ab} |R_\infty|}_{L^s}  + 
\norm{ |R_\infty|^{\ab+1}}_{L^s} 
\\
& \le  \norm{|K|-b}_{L^ {r_0}\cap L^\I} ^{\ab}   \norm{R_\infty}_{L^{r_0} \cap L^\I}  
+  \norm{R_\infty}_{L^{r_0} \cap L^\I} ^{\ab+1} \le C.
\end{align*}
Summing these estimates, we have 
\[
\norm{H(t)}_{L^\I} \le Ce^{-\frac 12 a v_*t}, \quad
\norm{H(t)}_{L^s} \le C.
\]
The lemma follows by H\"older inequality between $L^\I$ and $L^s$.
\end{proof}

\begin{proof}[Proof of Theorem \ref{th1.3}]
Fix a choice of $r_0$ satisfying \eqref{Ass3a}--\eqref{Ass3c} and
\eqref{Ass3d}.  Let $\chi_1 =\chi_1(x,t) = {\bf 1}_{x \le \frac
  12(v_0+v_1)t}$ and $\chi_2 = 1 - \chi_1$. Using \eqref{eq5.5}, we
have
\[
\norm{\chi_1 (W-b)}_{L^{r_0} \cap L^\I} + \norm{\chi_2 W}_{L^{r_0}
    \cap L^\I} \lesssim 1.
\]
Assume 
\begin{equation}
\label{eq5.6}  
\norm{\eta(t)}_{L^{r_2}} + \norm{\eta}_{\stt}\le e^{-c_1 v_*t}, \quad
  \forall t \ge 0.
\end{equation}
Note
\[
|f(W+\eta)-f(W)| \lesssim \chi_1 |W-b|^{\ab}|\eta| +\chi_2
|W|^{\alpha_1} |\eta| +|\eta|^{\ab+1}+|\eta|^{\alpha_1+1}+
|\eta|^{\alpha_2+1}.
\]
Thus  by \eqref{Ass3b} and \eqref{Ass3d} we have
\begin{align*}
 \norm{ f(W+\eta)-f(W)}_{L^{r_2'}}(\tau) &\le 
(1+\norm{\chi_1 (W-b)}^{\ab}_{L^{r_0} \cap L^\I} 
+ \norm{\chi_2 W}^{\alpha_1}_{L^{r_0} \cap L^\I})\norm{\eta}_{L^{2} \cap L^{r_2}}
\\
&\qquad + \norm{\eta}_{L^{2} \cap L^{r_2}}^{\ab+1}+
\norm{\eta}_{L^{2} \cap L^{r_2}}^{\alpha_1+1}+\norm{\eta}_{L^{2} \cap L^{r_2}}^{\alpha_2+1}
\\
&\le  C e^{-c_1 v_* \tau}.
\end{align*}
Denote the right side of \eqref{v600a2} as $\Phi \eta$.
The same argument as for Theorem \ref{th1.1} shows that
\[
\norm{\Phi \eta(t)}_{L^{r_2}} \le C v_*^{-1 + \theta} e^{-c_1 v_* t},\quad
\norm{\Phi \eta}_{\stt} \le C v_*^{-1 + \theta/2} e^{-c_1 v_* t}.
\]
Thus $\norm{\Phi \eta(t)}_{L^{r_2}} + \norm{\Phi \eta}_{\stt} \le
e^{-c_1 v_* t}$ for $v_*$ sufficiently large.

For the difference estimate, for $\eta_1$ and $\eta_2$ satisfying
\eqref{eq5.6}, we use
\[
|f(W+\eta_1)-f(W+\eta_2)| \lesssim \bke{\chi_1 |W-b|^{\ab} +\chi_2
|W|^{\alpha_1}  +\sum_{j=1,2} \bke{|\eta_j|^{\ab}+|\eta_j|^{\alpha_1}+
|\eta_j|^{\alpha_2}}}|\eta|
\]
where $\eta=\eta_1-\eta_2$, and follow the same argument for Theorem
\ref{th1.1} to derive, for $v_*$ sufficiently large,
\[
\norm{\Phi \eta_1 - \Phi \eta_2}\le \frac 12 \norm{ \eta_1 - \eta_2}
\]
where $\norm{\eta} = \sup _{t>0} e^{c_1 v_* t}\bke{ \norm{ \eta
    (t)}_{L^{r_2}} + \norm{ \eta }_{\stt}}$. We have shown that $\Phi$
is a contraction mapping in the class \eqref{eq5.6}.  The proof of
Theorem \ref{th1.3} is complete.
\end{proof}


\subsection{Proof of Theorem \ref{th1.4}}

In this section we prove Theorem \ref{th1.4} and use Proposition \ref{prop_main_1} to construct a train of
infinitely many solitons and a half-kink for space dimension 1.

We assume throughout this section that the assumptions of Theorem \ref{th1.4} hold. In particular, $(\phi_j,\omega_j)$ for $j\in \mathbb N$ denote a sequence of bound states satisfying assumptions \ref{T1}, \eqref{1.26} (with $v_\sharp$ to be determined later)
and $\phi_0=\phi_K$ is the kink profile given in Proposition \ref{th:kink} .

As in Section \ref{sec3.2}, our main task is to prove that the profile $W=K+R_\infty$ and the source term $H=f(W)-f(K)-\sum_{j\in\mathbb N} f(R_j)$ satisfy to the hypotheses of Proposition \ref{prop_main_1}. 

\begin{lemma}\label{lem:6.1}
Let $a\in(0,1)$. For $\lambda=a\min(1,2a)v_*/4>0$, we have 
\begin{align*}
\|W(t) \|_{\infty}+    e^{\lambda t} \|H(t) \|_2
 &\le C_1, &\forall\, t\ge 0;\\
\|\nabla W(t) \|_2+ \| \nabla W(t) \|_{\infty} + e^{\lambda  t} \| \nabla H(t)\|_2 &\le C(1+V_*),
&\forall\, t\ge 0.
\end{align*}
\end{lemma}

\begin{proof}
Since $R_\infty$ satisfies the same hypotheses as in Lemma \ref{th3.1}, we only have to treat the addition of the kink. 
We have, by Lemma \ref{th3.1} and Proposition \ref{th:kink}
\[
\norm{W}_\infty+\norm{\nabla W}_\infty\leq \norm{ K}_\infty+\norm{ R_\infty}_\infty+\norm{\nabla K}_\infty+\norm{\nabla R_\infty}_\infty\leq C.
\]
Note that by exponential decay $\nabla K\in L^2(\R)$, therefore, combined with Lemma \ref{th3.1} this gives
\[
\norm{\nabla W}_2 \leq \norm{\nabla  K}_2+\norm{\nabla R_\infty}_2\leq C.
\]
We now estimate the source term $H$. As in the proof of Lemma \ref{th3.1}, we fix $t>0$, take any $x\in\mathbb \R$ and choose $m=m(x)$ corresponding to the nearest profile, i.e.
\[
|x-v_mt|=\min_{j\in\mathbb N}|x-v_jt|.
\]
If $m\geq 1$, then as in the proof of Lemma \ref{th3.2B}, we still have 
\[
|(R_\infty-R_m)(t,x)|\leq Ce^{-\frac12av_*t},
\]
and by Proposition  \ref{th:kink} it holds
\[
|K(t,x)|\leq D_ae^{-a |x-v_0t|}\leq D_ae^{-\frac12av_*t}.
\]
Therefore, 
if $m\geq1$ we have
\[
H(t,x)\leq |f(R_\infty)-\sum_{j\in\mathbb N}f(R_j)|+A_4 |K|+|f(K)|\lesssim e^{-\frac12av_*t},
\]
where $A_4=\max_{s\in[0,\norm{W}_\infty]}f'(s)$. 
If $m=0$, we replace the previous estimate by 
\[
H(t,x)\leq A_4|R_\infty|+\sum_{j\in\mathbb N}|f(R_j)|\lesssim e^{-\frac12av_*t}.
\]
This implies that 
\[
\norm{H(t)}_\infty\lesssim  e^{-\frac12av_*t}.
\]

With $x$ and $m$ as above, if $m=0$, we have (using a similar expression as \eqref{4.5})
\begin{align*}
|\nabla H(t,x)|&\lesssim 
	(|f_z(K+R_\infty)-f_z(K)|+|f_{\bar z}(K+R_\infty)-f_{\bar z}(K)|)|\nabla K|+\sum_{j\in\N}|\nabla R_j|
\end{align*}


Since we are close to the kink ($m=0$), the last sum will be small :
\[
\sum_{j\in\N}|\nabla R_j|\lesssim e^{-\frac a4 v_* t}\sum_{j\in\N}\bka{v_j}\omega_j^{1/\alpha_1}e^{-\frac12\omega_j ^{1/2}|x-v_jt|}.
\]
In addition we have
\[
(|f_z(K+R_\infty)-f_z(K)|+|f_{\bar z}(K+R_\infty)-f_{\bar z}(K)|)\lesssim |R_\infty|\lesssim  e^{-\frac a4 v_* t}\sum_{j\in\N}\omega_j^{1/\alpha_1}e^{-\frac12\omega_j ^{1/2}|x-v_jt|}.
\]
Therefore 
\[
|\nabla H(t,x)| \lesssim  e^{-\frac a4 v_* t}\sum_{j\in\N}\bka{v_j}\omega_j^{1/\alpha_1}e^{-\frac12\omega_j ^{1/2}|x-v_jt|}.
\]
The estimate for the case $m\geq 1$ is similar as in Lemma  \ref{th3.1} and we can conclude by \eqref{Ass4b} that
\[
\norm{\nabla H}_2 \lesssim e^{-\lambda t}\sum_{j\in\N_0}\bka{v_j}\omega_j^{1/\alpha_1-d/4}\leq e^{-\lambda t}V_* .
\]

Let now $s$ be defined as in the proof of Lemma \ref{th3.1}. By \eqref{th1.4-eq1}, we can further assume
\begin{equation}\label{eq4.4}
r_0 \le s(\ab+1).
\end{equation}
For simplicity in notations, assume that the kink in not moving, i.e. $v_0=0$. Therefore the main contribution will come from the kink for $x<0$ and the soliton train for $x>0$. We have on the right
\begin{multline*}
\norm{H}_{L^s(x>0)}\leq \norm{f(K+R_\infty)-f(R_\infty)}_{L^s(x>0)}+\norm{f(K)}_{L^s(x>0)}+\norm{f(R_\infty)}_{L^s}+\sum_{j\in\N}\norm{f(R_j)}_{L^s}\\
\leq A_4 \norm{K}_{L^s(x>0)}+C<+\infty,
\end{multline*}
where the last inequality is due to exponential decay to $0$ on the right for the kink. On the left, we have 
\begin{equation*}
\norm{H}_{L^s(x<0)}\leq \norm{f(K+R_\infty)-f(K)}_{L^s(x<0)}+\sum_{j\in\N}\norm{f(R_j)}_{L^s}
\end{equation*}
The first term cannot be treated as previously (unless $R_\infty\in L^s(\R)$, which is a priori not the case). 
%
%
Since $f$ verifies \eqref{F1-2},   by the mean value theorem we  have 
\[
|f(K+R_\infty)-f(K)|\lesssim \left((|K-b|+|R_\infty| )^{\tilde\alpha} +(|K-b|+|R_\infty| )^{\alpha_2}\right)|R_\infty|
\]
Hence,
\[
\norm{f(K+R_\infty)-f(K)}_{L^s(x<0)}\lesssim \left(\norm{K-b}_{L^1(x<0)}^{\tilde\alpha}+\norm{K-b}_{L^1(x<0)}^{\alpha_2} \right) \norm{R_\infty}_{L ^\infty}+\norm{R_\infty}^{1+\tilde\alpha}_{L^{s(\tilde\alpha+1)}}
\]
The right hand side is finite since $K$ converges exponentially to $b$ and the $L^{s(\tilde\alpha+1)}$-norm of $R_\infty$ is finite thanks to our choice of $r_0$ and
\eqref{eq4.4}. In conclusion,
\[
\norm{H}_{L^s}\leq \norm{H}_{L^s(x<0)}+\norm{H}_{L^s(x>0)}<+\infty.
\]
By interpolation between $s<2$ and $\infty$  we get 
\[
\norm{H}_{L^2}\lesssim e^{-\lambda t}.
\]
This concludes the proof.

\end{proof}

\begin{proof}[Proof of  Theorem \ref{th1.4}]
By Lemma \ref{lem:6.1}, there exists $v_\sharp$ such that if $v_*>v_\sharp$, then the hypothesis \eqref{v600a1} of Proposition
\ref{prop_main_1} is satisfied under the assumptions of Theorem
\ref{th1.4}. The conclusion of the Theorem then follows immediately from the conclusion of Proposition \ref{prop_main_1}.
\end{proof}

\section*{Acknowledgments}
The authors are grateful to Dong Li for stimulating discussions at the origin of this work.
The research of S. Le Coz is supported in part by the french ANR
through project ESONSE.  The research of Tsai is supported in part by
NSERC grant 261356-13 (Canada).

\bibliographystyle{habbrv}
\bibliography{biblio}

\end{document}